\documentclass[12pt,psfig,reqno]{amsart}
\usepackage{}
\usepackage{epsfig}
\usepackage{graphicx}
\usepackage{verbatim}
\usepackage{amssymb}
\usepackage{latexsym}
\usepackage{amsmath}
\usepackage{amsfonts}
\usepackage{amsbsy}
\usepackage{amscd}
\usepackage[usenames]{color} 

\newtheorem{prop}{Proposition}[section]
\newtheorem{lem}[prop]{Lemma}

\newtheorem{defi}{Definition}[section]
\newtheorem{cor}[prop]{Corollary}
\newtheorem{thm}[prop]{Theorem}

\newtheorem{exam}{Example}[section]

\begin{document}
\title{Spectral property of self-affine measures on ${\mathbb R}^n$}

\author{Jing-Cheng Liu and Jun Jason Luo}
\address{Key Laboratory of High Performance Computing and Stochastic Information Processing (Ministry of Education of China), College of Mathematics and
Computer Science, Hunan Normal University, Changsha, Hunan 410081, China} \email{liujingcheng11@126.com}
\address{College of Mathematics and Statistics, Chongqing University, Chongqing, 401331, China}\email{jun.luo@cqu.edu.cn}

\thanks{The research is supported in part by the NNSF of China (No.11171100, No.11301175, No.11301322, No.11571104), the program for excellent talents in Hunan Normal University (No. ET14101), Specialized Research Fund for the Doctoral Program of Higher Education of China (20134402120007),  the Fundamental and Frontier Research Project of Chongqing (No.cstc2015jcyjA00035).}

\date{\today}

\begin{abstract}  We study spectral properties of the self-affine measure $\mu_{M,\mathcal {D}}$ generated by an expanding integer  matrix $M\in M_n(\mathbb{Z})$ and  a consecutive collinear digit set $\mathcal {D}=\{0,1,\dots,q-1\}v$ where $v\in \mathbb{Z}^n\setminus\{0\}$ and $q\ge 2$ is an integer. Some sufficient conditions for $\mu_{M,\mathcal {D}}$  to be a spectral measure or to have infinitely many orthogonal exponentials are given. Moreover, for some special cases, we can obtain a necessary and  sufficient condition on the spectrality of $\mu_{M,\mathcal {D}}$. Our study generalizes the one dimensional results proved by Dai, {\it et al.} (\cite{Dai-He-Lai_2013, Dai-He-Lau_2014}).
\end{abstract}

\maketitle

\section{Introduction}
Let $\mu$ be a Borel probability measure with compact support on $\mathbb{R}^n$, We call $\mu$ a \emph{spectral measure} if there exists a discrete set $\Lambda\subset\mathbb{R}^n$ such that the set of complex exponentials $E(\Lambda):=\{e^{2\pi i<\lambda,x>}:\lambda \in \Lambda\}$ forms an orthonormal basis for $L^2(\mu)$.  The set $\Lambda$ is called a \emph{spectrum} for $\mu$.  The interest of spectral measures was initialled by Fuglede \cite{Fuglede1974} and his famous conjecture: $\chi_\Omega dx$ is a spectral measure on ${\mathbb R}^n$ if and only if $\Omega$ is a translational tile.  Later it was proved that the conjecture is false in both directions on ${\mathbb R}^n$ for $n\ge 3$ (\cite{Tao2004}, \cite{Kolountzakis-Matolcsi1}, \cite{Kolountzakis-Matolcsi2}); but it is still open for $n=1,2$. Moreover, the problem of spectral measures is also very attractive when we replace the Lebesgue measure  $\mu$ by  fractal measures such as self-similar/affine measures.

Let $M\in M_n(\mathbb{Z})$ be an expanding matrix (i.e., all the moduli of eigenvalues of $M$ are strictly larger than one)  and $\mathcal {D}\subset \mathbb{Z}^n$ be a finite digit set. Then the maps
\begin{equation*}
\phi_d(x)= M^{-1}(x + d),\quad d\in \mathcal {D}
\end{equation*}
are contractive with respect to a suitable norm in $\mathbb{R}^n$ \cite{Lagarias-Wang_1996} and it is well known that there is a unique nonempty compact set $T:=T(M,\mathcal{D})$ satisfying the set-valued equation \cite{Hutchinson_1981}:
\begin{equation*}
T=\bigcup_{d\in \mathcal {D}} \phi_d(T).
\end{equation*}

Such $T$ is called the \textit{self-affine set} (or \textit{attractor}) of the {\it iterated function system (IFS)} $\{\phi_d\}_{d\in \mathcal {D}}$. A self-affine set  can be equipped with a unique invariant probability measure $\mu:= \mu_{M,\mathcal {D}}$ defined by
\begin{equation*}
\mu= \frac{1}{|\mathcal {D}|}\sum_{d\in \mathcal {D}}\mu\circ \phi_d^{-1},
\end{equation*}
and  $\mu_{M,\mathcal {D}}$ is supported on $T$. We call $\mu_{M,\mathcal {D}}$ a \emph{self-affine measure} \cite{Hutchinson_1981}. In particular,  if $M$ is a multiple of an orthonormal matrix,  $T$ and $\mu_{M,\mathcal {D}}$ are often called {\it self-similar set} and {\it self-similar measure}, respectively.

It is natural to ask whether $\mu_{M,\mathcal {D}}$ is a spectral measure.  Jorgenson and Pederson \cite{Jorgenson-Pederson_1985} first studied the spectral property of  certain Cantor measures. It was found that the $1/4$-Cantor measure $\mu_{1/4}$ on $\mathbb{R}$ is a spectral measure while the $1/3$-Cantor measure $\mu_{1/3}$ is not.
Hu and Lau \cite{Hu-Lau_2008} further investigated the spectrality of Bernoulli convolutions $\mu_{\rho}$, and observed that $L^2(\mu_{\rho})$ contains an infinite orthonormal set of exponential functions if and only if the contraction ratio $\rho$ is the $n$-th root of a fraction $p/q$ where $p$ is odd and $q$ is even. Recently, Dai \cite{Dai_2012} completely settled the problem that the only spectral Bernoulli convolutions are of the contraction ratio $1/2k$. As a generalization of the Bernoulli convolution, Dai, He and Lai \cite{Dai-He-Lai_2013} studied the self-similar measure generated by a positive integer $b$ and consecutive digits $\mathcal {D}=\{0,1,\dots, q-1\}$, it was proved that $L^2(\mu_{b,\mathcal {D}})$ has infinitely many orthogonal exponentials if and only if  $\gcd(q,b)>1$, and $\mu_{b,\mathcal {D}}$ is a spectral measure provided $q|b$.  After that, replacing the integer $b$ by any real number $b'>1$, Dai, He and Lau \cite{Dai-He-Lau_2014} showed  that $\mu_{b', {\mathcal D}}$ is a spectral measure if and only if $b'\in {\mathbb N}$ and $q| b'$. For more general cases such as Moran measures, we refer to  \cite{An-He-Lau_2015, An-He_2014}.

However, there are few results on the spectrality of self-affine measures (\cite{Dutkay-Han-Jorgensen_2009}-\cite{Dutkay-Jorgensen2007}, \cite{Li_2008}-\cite{Li-Wen_2012}), among which  Li (\cite{Li_2008}-\cite{Li-Wen_2012}) mainly  considered the spectrality or non-spectrality of some self-affine measures in low dimensions. For example, Li and Wen  \cite{Li-Wen_2012} studied the measure $\mu_{M,\mathcal {D}}$ generated by $M\in M_2(\mathbb{Z})$ and $\mathcal {D}=\{0,1\}v$, where $v\in\mathbb{Z}^2\setminus\{0\}$. They obtained that $\mu_{M,\mathcal {D}}$ admits an infinite orthonormal set if $\det(M)$ is even; if $v$ is the eigenvector of $M$ with eigenvalue $\ell$, then  $\mu_{M,\mathcal {D}}$ is a spectral measure if and only if $\ell$ is even.

Motivated by the above results, in this paper, we study the spectrality of self-affine measures on ${\mathbb R}^n$ with consecutive collinear digits. First we need a decomposition of integer matrices, which may be known but we can not find the reference, so we provide a proof in Section 3 (see Lemma \ref{th(3.6)}).

Let $M\in M_n(\mathbb{Z}), v\in \mathbb{Z}^n\setminus\{0\}$, and let $r$ be the rank of the set of vectors $\{v, Mv, \dots, M^{n-1}v\}$.  Then there exists a unimodular matrix $B\in M_n(\mathbb{Z})$ such that $Bv=(x_1,\dots x_r, 0,\dots,0)^t$ and
\begin{equation}\label{eq(1.1)}
BMB^{-1}=\begin{bmatrix}
 M_1& C \\
 0& M_2
\end{bmatrix}
\end{equation}
where $M_1\in M_r(\mathbb{Z})$, $M_2\in M_{n-r}(\mathbb{Z})$ and $C\in M_{r,n-r}(\mathbb{Z})$. By making use of this matrix decomposition, we have the main theorem of the paper.

\begin{thm} \label{th(1.1)}
Let $M\in M_n(\mathbb{Z})$ be an expanding matrix, $\mathcal {D}=\{0,1,\dots, q-1\}v$ be a digit set where $q\ge 2$ is an integer and $v\in\mathbb{Z}^n\setminus\{0\}$. If the set of  vectors $\{v, Mv, \dots, M^{n-1}v\}$ has rank $r$ and $M, M_1$ are expressed as in \eqref{eq(1.1)}, then

(i) $\mu_{M,\mathcal {D}}$ has infinitely many orthogonal exponentials if $\gcd(q,\det(M_1))>1$;

(ii) $\mu_{M,\mathcal {D}}$ is a spectral measure if $q|\det(M_1)$.
\end{thm}

In particular, if the characteristic polynomial of  $M_1$ is of the simple form $f(x)=x^r+c$, then we have the following necessary and sufficient  condition on  the spectrality of $\mu_{M,\mathcal {D}}$.

\begin{thm} \label{th(1.2)}
Under the same assumption of Theorem \ref{th(1.1)}.  If  the characteristic polynomial of  $M_1$ is $f(x)=x^r+c$, then

(i) $\mu_{M,\mathcal {D}}$ has infinitely many orthogonal exponentials if and only if $\gcd(q,\det(M_1))>1$;

(ii) $\mu_{M,\mathcal {D}}$ is a spectral measure if and only if $q|\det(M_1)$.
\end{thm}

According to the dependence of the set of vectors $\{v, Mv, \dots, M^{n-1}v\}$, we will prove the above theorems by two steps: $r=n$ (Theorems \ref{th(3.2)}, \ref{th(3.4)})  and $r<n$ (Theorem \ref{th(3.7)}). The main idea is to use the techniques of linear algebra and matrix analysis.

If the rank $r=1$, then $v$ becomes an eigenvector of $M$ and  $M_1=\ell$ with  characteristic polynomial $f(x)=x-\ell$, by Theorem \ref{th(1.2)},  the following corollary  is immediate.

\begin{cor}
Under the same assumption of Theorem \ref{th(1.1)}. If $r=1$, in this case $v$ is an eigenvector of $M$ and  $M_1=\ell$ is the corresponding eigenvalue, then

(i) $\mu_{M,\mathcal {D}}$ has infinitely many orthogonal exponentials if and only if $\gcd(q, \ell)> 1$;

(ii) $\mu_{M,\mathcal {D}}$ is a spectral measure if and only if  $q|\ell$.
\end{cor}

For the organization of the paper, we recall several basic concepts and lemmas in Section 2 and prove the main results orderly in Section 3.

\section{Preliminaries}
In this section, we give some  preliminary definitions and  lemmas that we need in proving our main results.

\begin{defi}
Let $M\in M_n(\mathbb{Z})$ be an expanding matrix. Let $\mathcal {D}$ and $S$ be finite subsets of $\mathbb{Z}^n$ with the same cardinality $q$. We say $(M^{-1}\mathcal {D}, S)$ is an integral compatible pair if the  matrix
\begin{equation*}
H=\frac{1}{\sqrt{q}}\left[e^{2\pi i <M^{-1}d,s>}\right]_{d\in \mathcal {D},s\in S}
\end{equation*}
is unitary, i.e., $H^*H=I$, where $H^{*}$ denotes the transposed conjugate of $H$. At this time, $(M,\mathcal {D},S)$ is called Hadamard triple.
\end{defi}

It is fairly easy to construct an infinite mutually orthogonal set of exponential functions using the Hadamard triple assumption. However, to check these exponentials form a Fourier basis for $L^2(\mu)$ is a much more difficult task. Jorgensen and Pedersen conjectured that all Hadamard triples will generate self-affine spectral  measures. The conjecture was solved in dimension one by Laba and Wang \cite{Laba-Wang_2002}. Under various additional conditions, it was also valid in the high dimensions (see \cite{Dutkay-Jorgensen2007}, \cite{Strichartz_1998}, \cite{Strichartz_2000}). Until recently, Dutkay, Haussermann and Lai \cite{Dutkay-Haussermann-Lai_2015} have completely proved the conjecture.

\begin{lem} [\cite{Dutkay-Haussermann-Lai_2015}] \label{th(2.1)}
Let $(M,\mathcal {D},S)$ be a Hadamard triple. Then the self-affine measure $\mu_{M,\mathcal {D}}$ is a spectral measure.
\end{lem}

The Fourier transform of  $\mu_{M,\mathcal {D}}$ plays a key role in studying the sepctrality of the measure. We recall the definition by
\begin{equation}\label {eq(2.2)}
\hat{\mu}_{M,\mathcal {D}}(\xi):=\int e^{2\pi i<x,\xi>}d\mu_{M,\mathcal {D}}(x)=\prod_{j=1}^\infty m_\mathcal {D}({M^{*}}^{-j}\xi), \quad \xi\in\mathbb{R}^n
\end{equation}
where $m_\mathcal {D}(\cdot)=\frac{1}{|\mathcal {D}|}\sum_{d\in \mathcal {D}}{e^{2\pi i\langle d,\cdot\rangle}}$ is the mask polynomial of the digit set $\mathcal D$. It is easy to see that $m_\mathcal {D}$ is a $\mathbb{Z}^n$-periodic function for $\mathcal {D}\subset \mathbb{Z}^n$. Let $Z_\mathcal {D}^n:=\{x\in [0, 1)^n : m_\mathcal {D}(x)=0\}$ be the zero set of $m_\mathcal {D}$ in $[0,1)^n$. Then we have the following useful criterion for the existence of infinitely many orthogonal systems.

\begin{lem} [\cite{Li_2010}] \label{th(2.3)}
Let $M\in M_n(\mathbb{Z})$ be an expanding matrix and $\mathcal {D}\subset \mathbb{Z}^n $ be a finite digit set, if there exist $\alpha\in Z_\mathcal {D}^n$ and $\ell\in \mathbb{N}$ such that $M^{*\ell}\alpha \in \mathbb{Z}^n$. Then there are infinitely many orthogonal exponentials $E(\Lambda)$ in $L^2(\mu_{M,\mathcal {D}})$ with $\Lambda\subseteq \mathbb{Z}^n$.
\end{lem}

Let $M$, $\widetilde{M}$ be $n\times n$ integer matrices, and the finite sets $\mathcal {D}, S, \widetilde{\mathcal {D}},\widetilde{S} $ be in $\mathbb{Z}^n$.
We say that two triples $(M,\mathcal {D},S)$ and $(\widetilde{M},\widetilde{\mathcal {D}},\widetilde{S})$ are \emph{conjugate} (through the matrix $B$) if there
exists an integer matrix $B$ such that $\widetilde{M}=B^{-1}MB$, $\widetilde{\mathcal {D}}= B^{-1}\mathcal {D}$ and $\widetilde{S}=B^*S$. The following lemma is trivial.

\begin{lem}\label{th(2.2)}
Suppose that $(M,\mathcal {D},S)$ and $(\widetilde{M},\widetilde{\mathcal {D}},\widetilde{S})$ are two conjugate triples, through the
matrix B. Then

(i) if $(M,\mathcal {D},S)$ is a Hadamard triple then so is $(\widetilde{M},\widetilde{\mathcal {D}},\widetilde{S})$;

(ii) $\mu_{M,\mathcal {D}}$ is a spectral measure with spectrum $\Lambda$ if and only if $\mu_{\widetilde{M},\widetilde{\mathcal {D}}}$ is a spectral measure with spectrum $B^*\Lambda$;

(iii)  $\mu_{M,\mathcal {D}}$ has infinitely many orthogonal exponentials $E(\Lambda)$ in $L^2(\mu_{M,\mathcal {D}})$ if and only if $\mu_{\widetilde{M},\widetilde{\mathcal {D}}}$ has infinitely many orthogonal exponentials $E(B^*\Lambda)$ in $L^2(\mu_{\widetilde{M},\widetilde{\mathcal {D}}})$.
\end{lem}

\begin{proof}
The proofs of (i) and (ii) can be found in \cite{Dutkay-Jorgensen2007} (or  \cite{Dutkay-Haussermann-Lai_2015}). In fact, it is easy to see that $E(\Lambda)$ is an orthogonal set in $L^2(\mu_{M,\mathcal {D}})$ if and only if $(\Lambda-\Lambda)\setminus \{0\}\subset {\mathcal Z}(\hat{\mu}_{M,\mathcal {D}}):=\{\xi\in{\mathbb R}^n: \hat{\mu}_{M,\mathcal {D}}(\xi)=0\}$. Let $\lambda_1, \lambda_2\in \Lambda$, then
\begin{eqnarray*}
m_{\widetilde{\mathcal D}}(\widetilde{M}^{*-j}B^*(\lambda_1-\lambda_2))
&=& m_{\widetilde{\mathcal D}}(B^*{M}^{*-j}(\lambda_1-\lambda_2)) \\
&=& \frac{1}{|\mathcal {D}|}\sum_{d\in \mathcal {D}}{e^{2\pi i\langle B^{-1}d,B^*{M}^{*-j}(\lambda_1-\lambda_2)\rangle}}
= m_{\mathcal D}({M}^{*-j}(\lambda_1-\lambda_2)).
\end{eqnarray*}
Hence by $\eqref{eq(2.2)}$, $\lambda_1-\lambda_2\in {\mathcal Z}(\hat{\mu}_{M,\mathcal {D}})$ if and only if $B^*\lambda_1-B^*\lambda_2\in {\mathcal Z}(\hat{\mu}_{\widetilde{M},\widetilde{\mathcal D}})$, proving (iii).
\end{proof}

\section{Main results}

In this section, we first consider the case that $\left\{v,Mv,\dots,M^{n-1}v\right\}$ are linearly independent for expanding matrix $M\in M_n(\mathbb{Z})$
and $v\in \mathbb{Z}^n\setminus \{0\}$. We will get the sufficient conditions for $\mu_{M,{\mathcal {D}}}$ to have infinitely many orthogonal exponentials or to be a spectral measure that only depend on the value of  $\det(M)$. The following simple lemma is an important tool in constructing a suitable conjugate Hadamard triple.

\begin{lem}\label{th(3.1)}
Let $M\in M_n(\mathbb{Z})$ be an integer matrix  with characteristic polynomial $f(x)=x^n+a_1x^{n-1}+\cdots+a_{n-1}x+a_n$ and $v=(x_1,\dots,x_n)^t\in \mathbb{Z}^n\setminus \{0\}$. If the set of vectors $\left\{v,Mv,\dots,M^{n-1}v\right\}$ is linearly independent, then there exists an integer matrix $B$, such that
\begin{align} \label{eq(3.1)}
        \widetilde{M}:=\begin{bmatrix}
-a_1&1&0&\cdots &0\\
-a_2&0&1&\cdots &0\\
\vdots & \vdots & \vdots & \ddots & \vdots \\
-a_{n-1}&0&0&\cdots &1\\
-a_n&0&0&\cdots &0\\
  \end{bmatrix}=B^{-1}MB
\end{align}
and $B^{-1}v=(0,\dots,0,1)^t$.
\end{lem}

\begin{proof}
Let $B=\left[M^{n-1}v, M^{n-2}v, \dots,  Mv,v\right]$, then $B$ is  invertible and  $B(0,\dots,0,1)^t=v$. Hence  $B^{-1}v=(0,\dots,0,1)^t$. To prove \eqref{eq(3.1)}, we only need to show that $B\widetilde{M}=MB=\left[M^{n}v, M^{n-1}v, \dots,  M^2v,Mv\right]$. Indeed,
\begin{eqnarray} \label{eq(3.2)}  \nonumber
        B\widetilde{M}&=&\left[M^{n-1}v, M^{n-2}v, \dots ,  Mv,v \right]\begin{bmatrix}
-a_1&1&0&\cdots &0\\
-a_2&0&1&\cdots &0\\
\vdots & \vdots & \vdots & \ddots & \vdots \\
-a_{n-1}&0&0&\cdots &1\\
-a_n&0&0&\cdots &0\\
  \end{bmatrix}\\ \nonumber
&=&\left[\sum_{i=1}^n-a_i M^{n-i}v, M^{n-1}v, \dots,  M^2v, Mv \right]\\
&=&\left[(-\sum_{i=1}^na_i M^{n-i})v, M^{n-1}v, \dots,  M^2v, Mv \right].
\end{eqnarray}
Since $f(x)$ is the characteristic polynomial of $M$, it follows that $f(M)=M^n+a_1M^{n-1}+\cdots+a_{n-1}M+a_nI=0$ and $M^n=-\sum_{i=1}^na_i M^{n-i}$. By \eqref{eq(3.2)}, we have $B\widetilde{M}=\left[M^nv, M^{n-1}v, \dots,  M^2v,Mv \right]=MB$.
\end{proof}

\begin{thm} \label{th(3.2)}
Let $\mathcal {D}=\{0,1,\dots, q-1\}v$ with $q\ge 2, v\in \mathbb{Z}^n\setminus \{0\}$, and let $M\in M_n(\mathbb{Z})$ be an expanding matrix.
If $\{v, Mv,\dots, M^{n-1}v\}$ is linearly independent, then

(i) $\mu_{M,\mathcal {D}}$ has infinitely many orthogonal exponentials if $\gcd(q,\det(M))> 1$;

(ii) $\mu_{M,\mathcal {D}}$ is a spectral measure if $q|\det(M)$.
\end{thm}

\begin{proof}
Suppose $f(x)=x^n+a_1x^{n-1}+\cdots+a_{n-1}x+a_n$ is the characteristic polynomial of $M$. Let $\widetilde{M}$ be as  in \eqref{eq(3.1)} and
$\widetilde{v}=(0,\dots,0,1)^t$.   By Lemmas \ref{th(2.2)} and \ref{th(3.1)}, it suffices to prove the theorem for the measure
$\mu_{\widetilde{M},\widetilde{\mathcal {D}}}$ which is generated by $\widetilde{M}$ and $\widetilde{\mathcal {D}}=\{0,1,\dots, q-1\}\widetilde{v}$.

(i) For $\xi=(\xi_1,\dots, \xi_n)^t\in {\mathbb R}^n$, note that
\begin{eqnarray*}
m_{\widetilde{\mathcal {D}}}(\xi)&=&\frac{1}{q}\sum\limits_{\widetilde{d}\in \widetilde{\mathcal {D}}}{e^{2\pi i\langle \widetilde{d},\xi\rangle}}
=\frac{1}{q}(1+e^{2\pi i\xi_n}+e^{2\pi i2\xi_n}+\cdots+e^{2\pi i(q-1)\xi_n}).
\end{eqnarray*}
Then
\begin{eqnarray*}
Z_{\widetilde{\mathcal {D}}}^n&:=&\{\xi\in [0, 1)^n : m_{\widetilde{\mathcal {D}}}(\xi)=0\}\\
&=&\{(\xi_1,\dots,\xi_{n-1},j/q)^t: \xi_1,\dots,\xi_{n-1}\in [0,1), j=1,\dots, q-1\}.
\end{eqnarray*}
Suppose $\gcd(q,\det(M))=s>1$, let  $\alpha:=(0,\dots,0,1/s)^t\in{\mathbb R}^n$, then $\alpha\in Z_{\widetilde{\mathcal {D}}}^n$ and
\begin{align*}
{\widetilde{M}}^*\alpha=\begin{bmatrix}
-a_1&-a_2&\cdots &-a_{n-1}&-a_n\\
1 &0&\cdots &0 &0\\
\vdots & \vdots &  \ddots & \vdots  & \vdots \\
0&0&\cdots &0 &0\\
0&0&\cdots &1 &0\\
  \end{bmatrix}\begin{bmatrix}
0\\
0\\
\vdots\\
0\\
\frac{1}{s}\\
\end{bmatrix}=\begin{bmatrix}
\frac{-a_n}{s}\\
0\\
\vdots\\
0\\
0\\
\end{bmatrix}.
\end{align*}

As $a_n=(-1)^n\det(M)=(-1)^n\det(\widetilde{M})$, then $\frac{-a_n}{s}\in \mathbb{Z}$ and ${\widetilde{M}}^*\alpha\in \mathbb{Z}^n$.
Lemma \ref{th(2.3)} implies that $\mu_{\widetilde{M},\widetilde{\mathcal {D}}}$ has infinitely many orthogonal exponentials.

(ii) If $q|\det(M)$, let $u=(\frac{-a_n}{q},0,\dots,0)^t$, $\widetilde{S}=\{0,1,\dots,q-1\}u$, then
$|\widetilde{D}|=|\widetilde{S}|=q$ and $\widetilde{S}\subset\mathbb{Z}^n$. Moreover,
\begin{align*}
{\widetilde{M}}^{-1}\widetilde{v}=\begin{bmatrix}
0&0&\cdots &0&\frac{-1}{a_n}\\
1 &0&\cdots &0 &\frac{-a_{1}}{a_n}\\
\vdots & \vdots &  \ddots & \vdots  & \vdots \\
0&0&\cdots &0 &\frac{-a_{n-2}}{a_n}\\
0&0&\cdots &1 &\frac{-a_{n-1}}{a_n}\\
  \end{bmatrix}\begin{bmatrix}
0\\
0\\
\vdots\\
0\\
1\\
\end{bmatrix}=\begin{bmatrix}
\frac{-1}{a_n}\\
\frac{-a_{1}}{a_n}\\
\vdots\\
\frac{-a_{n-2}}{a_n}\\
\frac{-a_{n-1}}{a_n}\\
\end{bmatrix}.
\end{align*} It yields that
\begin{align*}
H=\frac{1}{\sqrt{q}}\left[e^{2\pi i \langle \widetilde{M}^{-1}\widetilde{d},\widetilde{s}\rangle}\right]_{\widetilde{d}\in \widetilde{\mathcal {D}}, \widetilde{s}\in\widetilde{ S}}
=\frac{1}{\sqrt{q}}\left[e^{2\pi i \frac{k\ell}{q}}\right]_{k,\ell\in \{0,1,\dots,q-1\} }
\end{align*}
is unitary. Hence $(\widetilde{M},\widetilde{\mathcal {D}},\widetilde{S})$ is a Hadamard triple, and $\mu_{\widetilde{M},\widetilde{\mathcal {D}}}$ is a spectral measure by Lemma \ref{th(2.1)}.
\end{proof}

In particular, if the characteristic polynomial of $M$ is of the simple form: $f(x)=x^n+c$, we obtain a sufficient and necessary condition for  $\mu_{M,\mathcal {D}}$ to have infinitely many orthogonal exponentials or to be a spectral measure. The following lemma was due to Dai, He and Lau \cite{Dai-He-Lau_2014}.

\begin{lem}\label{th(3.3)}
Let $\mu=\mu_1\ast\mu_2$ be the convolution of two probability measures $\mu_i$, $i=1, 2$, and they are not Dirac measures. Suppose that $E(\Lambda)$ is an orthogonal set of $\mu_1$ with $0\in\Lambda$, then $E(\Lambda)$ is also an orthogonal set of $\mu$, but cannot be a spectrum of $\mu$.
\end{lem}

\begin{thm}\label{th(3.4)}
Under the same assumption of Theorem \ref{th(3.2)} and the matrix $M$ with characteristic polynomial $f(x)=x^n+c$, then

(i) $\mu_{M,\mathcal {D}}$ has infinitely many orthogonal exponentials if and only if $\gcd(q, \det(M))>1$;

(ii) $\mu_{M,\mathcal {D}}$ is a spectral measure if and only if $q|\det(M)$.
\end{thm}

\begin{proof}
Likewise Theorem \ref{th(3.2)}, we consider the conjugate situation where the matrix
\begin{align*}
        \widetilde{M}:=\begin{bmatrix}
0&1&0&\cdots &0\\
0&0&1&\cdots &0\\
\vdots & \vdots & \vdots & \ddots & \vdots \\
0&0&0&\cdots &1\\
-c&0&0&\cdots &0\\
  \end{bmatrix}=B^{-1}MB
\end{align*}
and digit set $\widetilde{\mathcal {D}}=\{0,1,\dots, q-1\}\widetilde{v}$ where $\widetilde{v}=(0,\dots, 0, 1)^t$  as  in \eqref{eq(3.1)}. Then the sufficiencies of (i) and (ii) come from Theorem \ref{th(3.2)}, we now prove the necessities.

(i) Let $\mathcal {Z}_0:=\{\xi=(\xi_1,\dots, \xi_n)^t\in \mathbb{R}^n: m_{\widetilde{\mathcal D}}(\xi)=0\}$ be the zero set of the mask function $m_{\widetilde{\mathcal D}}$. Then
\begin{align*}
 \mathcal {Z}_0=\left\{\left( \xi_1,\dots,\xi_{n-1},k+\frac{i}{q}\right)^t \in{\mathbb R}^n: \xi_1,\dots, \xi_{n-1} \in \mathbb{R}, k\in\mathbb{Z}, 1\leq i\leq q-1\right\}.
\end{align*}

For $1\leq j\leq n $, by letting $\xi_n=k+\frac{i}{q}$ as above, we observe that
\begin{eqnarray*}
\mathcal {Z}_j:=\widetilde{M}^{*j}(\mathcal {Z}_0)&=& \left\{\left( \begin{matrix}
-c\xi_{n-j+1}\\
\vdots \\
-c\xi_{n-1}\\
-c\xi_n\\
\xi_1\\
\vdots \\
\xi_{n-j}
  \end{matrix}\right): \xi_1,\dots, \xi_{n-1} \in \mathbb{R},\xi_n=k+\frac{i}{q}, k\in {\mathbb Z}, 1\le i\le q-1\right\}\\
  &\subset&\left\{\left( \begin{matrix}
\xi'_1\\
\vdots \\
\xi'_{j-1}\\
k'-\frac{ci}{q}\\
\xi'_{j+1}\\
\vdots \\
\xi'_{n}
  \end{matrix}\right): \xi'_1,\dots,\xi'_{j-1}, \xi'_{j+1},\dots, \xi'_{n} \in \mathbb{R}, k'\in\mathbb{Z}, 1\leq i\leq q-1\right\}.
\end{eqnarray*}

By (\ref{eq(2.2)}), the zero set ${\mathcal Z}(\widehat{\mu}_{\widetilde{M},\widetilde{\mathcal {D}}})$ of the Fourier transform $\widehat{\mu}_{\widetilde{M},\widetilde{\mathcal {D}}}$ can be written as
$${\mathcal Z}(\widehat{\mu}_{\widetilde{M},\widetilde{\mathcal {D}}})=\bigcup_{j=1}^\infty\mathcal {Z}_j.$$

If $\gcd(q, \det(M))=1$, it follows from $c=(-1)^n\det(M)$ that
$$
\{k'-\frac{ci}{q}:k'\in\mathbb{Z}, 1\leq i\leq q-1\}\subset \{k+\frac{i}{q}: k\in\mathbb{Z}, 1\leq i\leq q-1\}.
$$
Hence $\mathcal {Z}_n\subset \mathcal {Z}_0$, which shows that
 \begin{align} \label{eq(3.5)}
{\mathcal Z}(\widehat{\mu}_{\widetilde{M},\widetilde{\mathcal {D}}})=\bigcup_{j=1}^n\mathcal {Z}_j.
\end{align}

We now prove that there are only finite mutually orthogonal  exponentials in $L^2(\mu_{\widetilde{M},\widetilde{\mathcal {D}}})$. If otherwise, there exists a mutually orthogonal set $E(\Lambda)$ with $\Lambda=\{\tau_\ell\}_{\ell=1}^\infty$ for $\mu_{\widetilde{M},\widetilde{\mathcal {D}}}$, we may assume $\tau_1=0$ so that $\Lambda\setminus\{0\}\subset {\mathcal Z}(\widehat{\mu}_{\widetilde{M},\widetilde{\mathcal {D}}})$.  By \eqref{eq(3.5)}, there exists  $1\leq j_0\leq n$ such that $\mathcal {Z}_{j_0}$ contains  infinitely many elements of $\Lambda$. Without loss of generality, we let $\mathcal {Z}_{1}$ contain infinite elements of $\Lambda$, say $\{\tau_{\ell_i}\}_{i=1}^\infty$. According to the form of ${\mathcal Z}_1$ above, there exist an $i_0\in\{1,\dots, q-1\}$ and a subsequence $\{\tau_{\ell_{i_j}}\}_{j=1}^\infty$ satisfying
 \begin{align*}
 \left\{\tau_{\ell_{i_j}}\right\}_{j=1}^\infty\subset\left\{\left(k-\frac{ci_0}{q},\xi_2, \dots, \xi_{n} \right)^t: \xi_2, \dots, \xi_{n}\in \mathbb{R}, k\in \mathbb{Z}\right\}.
\end{align*}
Then for $j\geq 2$, the difference
\begin{align*}
\tau_{\ell_{i_j}}-\tau_{\ell_{i_1}}\in \left\{\left(k_1,\xi'_2, \dots, \xi'_{n} \right)^t: \xi'_2, \dots, \xi'_{n}\in \mathbb{R}, k_1\in \mathbb{Z}\right\}.
\end{align*}
Hence $\tau_{\ell_{i_j}}-\tau_{\ell_{i_1}}\notin\mathcal {Z}_{1}$ for all $j\ge 2$. Since $(\Lambda-\Lambda)\setminus\{0\}\subset {\mathcal Z}(\widehat{\mu}_{\widetilde{M},\widetilde{\mathcal {D}}})$, it follows that
 \begin{align*}
 \left\{\tau_{\ell_{i_j}}-\tau_{\ell_{i_1}}\right\}_{j=2}^\infty\subset\bigcup_{j=2}^n\mathcal {Z}_j.
\end{align*}

By continuing  the above process $n$ times, finally we obtain an infinite sequence of differences $\left\{\tau_{\ell_{m}}-\tau_{\ell^*}\right\}_{m=1}^\infty$ such that
\begin{align*}
\tau_{\ell_{m}}-\tau_{\ell^*}\in\left\{\left(k_1,k_2, \dots, k_n \right)^t: k_i\in \mathbb{Z}, i=1,\dots, n\right\}.
\end{align*}
That means $\tau_{\ell_{m}}-\tau_{\ell^*}\notin {\mathcal Z}(\widehat{\mu}_{\widetilde{M},\widetilde{\mathcal {D}}})$, which is impossible. Therefore, there are only finite  mutually orthogonal exponentials, and $\mu_{\widetilde{M},\widetilde{\mathcal {D}}}$ is not a spectral measure.

(ii)  As $c=(-1)^n\det(M)$, we only need to prove the $\mu_{\widetilde{M},\widetilde{\mathcal {D}}}$ is a non-spectral measure for $\gcd(q, c)=d$ where $1<d<q$. Let $q=q'd$, $c=c'd$, then $\gcd(c',q')=1$.  Denote by
$$\widetilde{\mathcal {D}}=\{0,1,\dots, d-1\}\widetilde{v}\oplus d\{0,1,\dots, q'-1\}\widetilde{v}:=\widetilde{\mathcal {D}}_1\oplus\widetilde{\mathcal {D}}_2.$$ Let
$$
\mu_1=\delta_{\widetilde{M}^{-1}\widetilde{\mathcal {D}}}\ast\cdots\ast\delta_{\widetilde{M}^{-n}\widetilde{\mathcal {D}}}\ast\delta_{\widetilde{M}^{-(n+1)}\widetilde{\mathcal {D}}_1}\ast\delta_{\widetilde{M}^{-(n+2)}\widetilde{\mathcal {D}}}\ast\cdots \ \text{and} \quad \mu_2=\delta_{\widetilde{M}^{-(n+1)}\widetilde{\mathcal {D}}_2}
$$ where $\delta_E=\frac{1}{\# E}\Sigma_{e\in E}\delta_e$ for a finite set $E$ and Dirac mass measure $\delta_e$ at the point $e$, and the infinite convolutions converge in the weak sense. Then we have  $\mu_{\widetilde{M},\widetilde{\mathcal {D}}}=\mu_1\ast\mu_2$.

It can be seen that
$${\mathcal Z}(\widehat{\delta}_{\widetilde{M}^{-1}\widetilde{\mathcal {D}}})=\{(-c\frac{m}{q}, \eta_1,\dots,\eta_{n-1})^t: m\in\mathbb{Z},
q\nmid m, \eta_1,\dots,\eta_{n-1}\in \mathbb{R}\}$$
and
$${\mathcal Z}(\widehat{\mu}_2)=\{(c^2\frac{m'}{dq'}, \eta'_1,\dots,\eta'_{n-1})^t: m'\in\mathbb{Z},
q'\nmid m', \eta'_1,\dots,\eta'_{n-1}\in \mathbb{R}\}.$$
Write $c^2\frac{m'}{dq'}=-c\frac{-cm'}{q}$. Since $\frac{-cm'}{q}=\frac{-cm'}{dq'}=\frac{-c'm'}{q'}$, $\gcd(c',q')=1$ and $q'\nmid m'$, it follows that  $q\nmid -cm'$, and ${\mathcal Z}(\widehat{\mu}_2)\subset {\mathcal Z}(\widehat{\delta}_{\widetilde{M}^{-1}\widetilde{\mathcal {D}}})$. Hence $\mathcal {Z}(\widehat{\mu}_{\widetilde{M},\widetilde{\mathcal {D}}})= {\mathcal Z}(\widehat{\mu}_1)$. This shows that any orthogonal set $E(\Lambda)$ of  $\mu_{\widetilde{M},\widetilde{\mathcal {D}}}$  is an orthogonal set of $\mu_1$, proving that $\mu_{\widetilde{M},\widetilde{\mathcal {D}}}$ is a non-spectral measure by Lemma \ref{th(3.3)}.
\end{proof}

\begin{exam}\label{ex(3.1)}
Let \begin{eqnarray*}
M=\begin{bmatrix}
2&6&4\\
-1&2&2\\
-1&-1&-4
\end{bmatrix}
\end{eqnarray*}  and $\mathcal {D}=\{0,1,\dots,q-1\}v$ where $q\ge 2, v=(0,0,1)^t$. Then $M$ is an expanding matrix and $\{v, Mv, M^2v\}$ is linearly independent. Moreover, the characteristic polynomial of $M$ is $f(x)=x^3+36$. Hence by Theorem \ref{th(3.4)}, we have

(i) $\mu_{M,\mathcal {D}}$ has infinitely many orthogonal exponentials if and only if $\gcd(36,q)>1$;

(ii) $\mu_{M,\mathcal {D}}$ is a spectral measure if and only if $q|36$.
\end{exam}

On the other hand, if the set $\{v, Mv, \dots, M^{n-1}v\}$ is linearly dependent, by using some techniques of linear algebra, we can reduce the dimension and get the similar results as Theorems \ref{th(3.2)} and \ref{th(3.4)} .

\begin{lem} \label{th(3.6)}
Let $v\in \mathbb{Z}^n\setminus \{0\}$ and  $M\in M_n(\mathbb{Z})$. If $\{v, Mv,\dots, M^{n-1}v\}$ is linearly dependent with rank $r<n$, then there exists a unimodular matrix $B\in M_n(\mathbb{Z})$ such that $Bv=(x_1,\dots, x_r, 0,\dots,0)^t$ and
\begin{equation}\label{matrixform}
\widetilde{M}=BMB^{-1}=\begin{bmatrix}
 M_1& C \\
 0& M_2
\end{bmatrix}
\end{equation}
where $M_1\in M_r(\mathbb{Z})$, $M_2\in M_{n-r}(\mathbb{Z})$ and $C\in M_{r,n-r}(\mathbb{Z})$.
\end{lem}

\begin{proof}
Let $A=\left[M^{r-1}v,\dots, Mv, v\right]=\left[a_{ij}\right]_{1\leq i\leq n, 1\leq j\leq r}$.  It is well known that by a series of elementary (unimodular) row operations on $A$  which is equivalent to multiple a unimodular $B$ from the left side of $A$, we can get $BA=\begin{bmatrix} A'_{r,r}\\  0_{n-r,r} \end{bmatrix} \in M_{n,r}(\mathbb{Z})$, i.e.,
\begin{align}\label{eq(3.9)}
 BA=B\begin{bmatrix}
a_{11}&a_{12}&\cdots &a_{1r}\\
a_{21}&a_{22}&\cdots &a_{2r}\\
\vdots & \vdots & & \vdots \\
a_{n1}&a_{n2}&\cdots &a_{nr}
  \end{bmatrix}=\begin{bmatrix}
a'_{11}&a'_{12}&\cdots &a'_{1r}\\
0&a'_{22}&\cdots &a'_{2r}\\
\vdots & \vdots & \ddots & \vdots \\
0&0&\cdots &a'_{rr}\\
0&0&\cdots &0\\
\vdots & \vdots &  & \vdots \\
0&0&\cdots &0
  \end{bmatrix}.
\end{align}

Now we show that  the matrix $B$ is the desired one.  As $v=(a_{1r},a_{2r},\dots, a_{nr})^t$, it follows from \eqref{eq(3.9)} that
$Bv=(a'_{1r},\dots a'_{rr},0,\dots,0)^t:=(x_1,\dots, x_r, 0,\dots,0)^t$.

Suppose
\begin{equation*}
BMB^{-1}=\begin{bmatrix}
 M_{r,r}& C_{r,n-r} \\
 D_{n-r,r}& N_{n-r,n-r}
\end{bmatrix}.
\end{equation*}
Then
\begin{equation*}
BM=\begin{bmatrix}
 M_{r,r}& C_{r,n-r} \\
 D_{n-r,r}& N_{n-r,n-r}
\end{bmatrix}B.
\end{equation*}
Multiplying $A=[M^{r-1}v,\dots, Mv, v]$ from both right sides of the above identity, we have
\begin{eqnarray}\label{eq(3.8.0)} \nonumber
BMA&=&\begin{bmatrix}
 M_{r,r}& C_{r,n-r} \\
 D_{n-r,r}& N_{n-r,n-r}
\end{bmatrix}BA\\
&=& \begin{bmatrix}
 M_{r,r}& C_{r,n-r} \\
 D_{n-r,r}& N_{n-r,n-r}
\end{bmatrix}\begin{bmatrix} A'_{r,r}\\  0_{n-r,r} \end{bmatrix}
=\begin{bmatrix}  M_{r,r}A'_{r,r}\\  D_{n-r,r}A'_{r,r} \end{bmatrix}.
\end{eqnarray}

It is known that $\{M^{r-1}v,M^{r-2}v,\dots,Mv,v\}$ is a maximal linear independent group as the rank of $\{M^{n-1}v,M^{n-2}v,\dots,Mv,v\}$ is $r<n$. Then there exist $b_0,b_1,\dots,b_{r-1}\in \mathbb{R}$ such that
$$
M^{r}v=b_0v+b_1Mv+\cdots +b_{r-1}M^{r-1}v
$$
and by \eqref{eq(3.9)},
$$
BM^{r}v=(c_1,\dots, c_r,0,\dots,0)^t
$$
where $c_r=b_0a'_{rr}\neq 0$. Moreover,
\begin{eqnarray}\label{eq(3.9.0)} \nonumber
BMA&=&BM[M^{r-1}v,\dots, Mv, v]=[BM^{r}v,\dots, BM^2v, BMv]\\
&=&\begin{bmatrix}
c_1&a'_{11}& \cdots &a'_{1r-1}\\
\vdots & \vdots & \ddots & \vdots \\
c_{r-1}&0&\cdots &a'_{r-1r-1}\\
c_r&0&\cdots &0\\
0&0&\cdots &0\\
\vdots & \vdots &  & \vdots \\
0&0&\cdots &0
  \end{bmatrix}:= \begin{bmatrix}  Q_{r,r}\\  0_{n-r,r} \end{bmatrix}.
\end{eqnarray}

From \eqref{eq(3.8.0)} and \eqref{eq(3.9.0)}, it follows that $D_{n-r,r}A'_{r,r}=0$, and $D_{n-r,r}=0$ as $A'_{r,r}$ is invertible.  Thus we have
\begin{equation*}
BMB^{-1}=\begin{bmatrix}
 M_{r,r}& C_{r,n-r} \\
 0& N_{n-r,n-r}
\end{bmatrix}.
\end{equation*}
Since $B$ is a unimodular matrix, $B^{-1}$ and $BMB^{-1}$ are also integer matrices.  Therefore $M_{r,r}$, $N_{n-r,n-r}$ and $C_{r,n-r}$
are integer matrices as desired.
\end{proof}

\begin{exam}\label{ex(3.1)}
Let \begin{eqnarray*}
M=\begin{bmatrix}
1&-3&3\\
3&-5&3\\
6&-6&4
\end{bmatrix},  \ \ \  B=\begin{bmatrix}
1&0&0\\
-1&1&0\\
-2&0&1
\end{bmatrix}
\end{eqnarray*} and $v=(1,1,2)^t$, then $Mv=4v$, $Bv=(1,0,0)^t$ and
\begin{eqnarray*}
BMB^{-1}=\begin{bmatrix}
1&0&0\\
-1&1&0\\
-2&0&1
\end{bmatrix}\begin{bmatrix}
1&-3&3\\
3&-5&3\\
6&-6&4
\end{bmatrix}\begin{bmatrix}
1&0&0\\
1&1&0\\
2&0&1
\end{bmatrix}=\begin{bmatrix}
4&-3&3\\
0&-2&0\\
0&0&-2
\end{bmatrix}.
\end{eqnarray*}
\end{exam}

\begin{thm} \label{th(3.7)}
Let  $M\in M_n({\mathbb Z})$ be an expanding matrix  and ${\mathcal D}=\{0,1,\dots, q-1\}v$ where $q\ge 2$ and  $v\in {\mathbb Z}^n\setminus\{0\}$. If $\{v, Mv,\dots, M^{n-1}v\}$ is linearly dependent with rank $r<n$, and $M, M_1$ are expressed as in \eqref{matrixform}, then

(i) $\mu_{M,\mathcal {D}}$ has infinitely many orthogonal exponentials  if $\gcd(q, \det(M_1))>1$;

(ii) $\mu_{M,\mathcal {D}}$ is a spectral measure if $q|\det(M_1)$, where $M_1$ is as in Lemma \ref{th(3.6)};\\
In particular, if the characteristic polynomial of  $M_1$ is $f(x)=x^r+c$, then

(iii) $\mu_{M,\mathcal {D}}$ has infinitely many orthogonal exponentials  if and only if  $\gcd(q, \det(M_1))>1$;

(iv) $\mu_{M,\mathcal {D}}$ is a spectral measure if and only if $q|\det(M_1)$.

\end{thm}

\begin{proof}
By Lemma \ref{th(3.6)}, there exists a unimodular matrix $B\in M_n(\mathbb Z)$ such that $Bv=(x_1,\dots, x_r, 0,\dots,0)^t:=\widetilde {v}$ and
\begin{equation*}\label{eq:A_D_star}
\widetilde{M}=BMB^{-1}=\begin{bmatrix}
 M_1& C \\
 0& M_2
\end{bmatrix}.
\end{equation*}
Let $\widetilde{\mathcal {D}}=B{\mathcal D}=\{0,1,\dots,q-1\}\widetilde {v}$. Hence it suffices to consider the conjugate case for $\widetilde{M}$ and $\widetilde{\mathcal {D}}$.  For $j=1,2,\dots$, we have
\begin{equation*}
\widetilde{M}^{*-j}=\begin{bmatrix}
 {M_1^*}^{-j}& 0 \\
\times &  {M_2^*}^{-j}
 \end{bmatrix}.
\end{equation*}
Let $v'=(x_1,\dots,x_r)^t$ and $\mathcal {D}'=\{0,1,\dots,q-1\}v'$. For $\xi=(\xi_1,\dots,\xi_n)^t\in {\mathbb R}^n$, we denote by $\xi'=(\xi_1,\dots,\xi_r)^t$.  It follows from \eqref{eq(2.2)} that
\begin{eqnarray*}
\hat{\mu}_{\widetilde{M},\widetilde{\mathcal {D}}}(\xi)&=&\prod_{j=1}^\infty\frac{1}{|\widetilde{D}|}\sum\limits_{d\in \widetilde{D}}{e^{2\pi i\langle d,{\widetilde{M^{*}}}^{-j}\xi\rangle}}=\prod_{j=1}^\infty\frac{1}{q}\sum_{k=0}^{q-1}{e^{2\pi i\langle k\widetilde{v},{\widetilde{M^{*}}}^{-j}\xi\rangle}}\\
&=&\prod_{j=1}^\infty\frac{1}{q}\sum_{k=0}^{q-1}{e^{2\pi i\langle kv', {M_1^*}^{-j}\xi'\rangle}}=\hat{\mu}_{M_1,{\mathcal {D}}'}(\xi').
\end{eqnarray*}
If $E(\Lambda)$ is an  orthogonal system or an orthonormal basis for $L^2(\mu_{\widetilde{M},\widetilde{\mathcal {D}}})$, we let
$\Lambda_r=\{\lambda'=(\lambda_1,\dots, \lambda_r)^t: \lambda=(\lambda_1,\dots, \lambda_n)^t\in \Lambda\}$, then $E(\Lambda_r)$ is an orthogonal system or  an orthonormal basis for $L^2(\mu_{M_1,\mathcal {D}'})$. Conversely, if $E(\Lambda_r)$ is an orthogonal system or  an orthonormal basis for $L^2(\mu_{M_1,\mathcal {D}'})$, we let
$$
\Lambda:=\left\{\left(\begin{matrix}
\lambda'\\
\varphi(\lambda')
 \end{matrix}\right):\lambda'\in\Lambda_r \right\}
$$
where $\varphi:\Lambda_r \rightarrow \mathbb{R}^{n-r}$ is an arbitrary single-valued function, then $E(\Lambda)$ is an orthogonal system or  an orthonormal basis for $L^2(\mu_{\widetilde{M},\widetilde{\mathcal {D}}})$. Therefore, $\mu_{M_1,\mathcal {D}'}$ and $\mu_{\widetilde{M},\widetilde{\mathcal {D}}}$ have the same spectrality.

By the assumption, $\{v', M_1v', \dots, M_1^{r-1}v'\}$ is linearly independent. It follows from Theorems \ref{th(3.2)} and \ref{th(3.4)} that  (i)-(iv) hold for $\mu_{M_1,\mathcal {D}'}$, hence hold for $\mu_{M,\mathcal {D}}$.
\end{proof}

If the rank of $\{v, Mv,\dots, M^{n-1}v\}$ is $r=1$, then $v$ becomes an eigenvector of $M$ and  $M_1=\ell$ is the corresponding eigenvalue. Indeed, from Lemma \ref{th(3.6)}, it can be seen that $Mv=B^{-1}BMB^{-1}Bv=B^{-1}\widetilde{M}Bv=B^{-1}(\ell x_1,0,\cdots,0)^t=\ell B^{-1}Bv=\ell v$.  By Theorem \ref{th(3.7)},  the following corollary is immediate.

\begin{cor}\label{th(3.8)}
Under the same assumption as Theorem \ref{th(3.7)} with the rank $r=1$, in this case $v$ is an eigenvector of $M$ and  $M_1=\ell$ (as in (\ref{matrixform})) is the corresponding eigenvalue, then

(i) $\mu_{M,\mathcal {D}}$ has infinitely many orthogonal exponentials if and only if $\gcd(q,\ell)>1$;

(ii) $\mu_{M,\mathcal {D}}$ is a spectral measure if and only if  $q|\ell$.
\end{cor}

\begin{exam}\label{ex(3.2)}
Let $M, B$ and $v$ as in Example \ref{ex(3.1)}. Let $\mathcal {D}_1=\{0,1,2,3,4,5\}v$ and $\mathcal {D}_2=\{0,1\}v$, then $\mu_{M,\mathcal {D}_1}$ has infinitely many orthogonal exponentials but not a spectral measure and $\mu_{M,\mathcal {D}_2}$ is a spectral measure by Corollary \ref{th(3.8)}.
\end{exam}

\bigskip
\noindent {\bf Acknowledgements:} Part of this work was done when the  authors were visiting the Chinese University of Hong Kong, they would like to thank  Professor Ka-Sing Lau for his support and valuable discussions.


\begin{thebibliography}{999}
\bibitem{An-He-Lau_2015} L.X. An, X.G. He, K.S. Lau, Spectrality of a class of infinite convolutions, {\em Adv. Math.}, {\bf 283} (2015), 362--376.

\bibitem{An-He_2014} L.X. An, X.G. He, A class of spectral Moran measures,{ \em J. Funct. Anal.}, {\bf 266} (2014), 343--354.

\bibitem{Dai_2012} X.R. Dai, When does a Bernoulli convolution admit a spectrum? {\em Adv. Math.}, {\bf 231} (2012), 1681--1693.

\bibitem{Dai-He-Lai_2013} X.R. Dai, X.G. He, C.K. Lai, Spectral structure of Cantor measures with consecutive digits, {\em Adv. Math.}, {\bf 242} (2013), 187--208.


\bibitem{Dai-He-Lau_2014} X.R. Dai, X.G. He, K.S Lau, On spectral N-Bernoulli measures, {\em Adv. Math.}, {\bf 259} (2014), 511--531.

\bibitem{Dutkay-Han-Jorgensen_2009}
D. Dutkay, D. Han, P. Jorgensen, Orthogonal exponentials, translations and Bohr completions, {\em J. Funct. Anal.}, {\bf 257}(2009),2999--3019.

\bibitem{Dutkay-Haussermann-Lai_2015} D.E. Dutkay, J. Haussermann,  C.K. Lai,  Hadamard triples  generate self-affine spectral measure, arxiv.org/pdf/1506.01503.pdf.


\bibitem{Dutkay-Jorgensen2007} D.E. Dutkay and P.E.T. Jorgensen, Fourier frequencies in affine iterated function systems, {\em J. Funct. Anal.}, {\bf 247(1)}(2007) 110--137.

\bibitem{Fuglede1974} B. Fuglede, Commuting self-adjoint partial differential operators and a group theoretic problem, {\em J. Funct. Anal.}, {\bf 16} (1974) 101--121.

\bibitem{Hu-Lau_2008} T.Y. Hu, K.S. Lau, Spectral property of the Bernoulli convolutions, {\em Adv. Math.}, {\bf 219} (2008), 554--567.

\bibitem{Hutchinson_1981} J. E. Hutchinson, Fractals and self-similarity, {\em Indiana Univ. Math. J.}, {\bf 30} (1981), 713--747.


\bibitem{Jorgenson-Pederson_1985}
P.E.T. Jorgenson and S. Pederson, Dense analytic subspaces in fractal $L^2$-spaces, {\em J. Anal. Math.}, {\bf 75} (1998),185--228.

\bibitem{Kolountzakis-Matolcsi1} M. Kolountzakis and M. Matolcsi,  Complex Hadamard matrices and the spectral set conjecture, {\em Collect. Math.} (2006) 661--671.

\bibitem{Kolountzakis-Matolcsi2} M. Kolountzakis and M. Matolcsi,  Tiles with no spectra, {\em Forum Math.}, {\bf 18} (2006), 519--528.

\bibitem{Laba-Wang_2002} I. Laba and Y. Wang, On spectral Cantor measures, {\em J. Funct. Anal.}, {\bf 193} (2002), 409--420.

\bibitem{Lagarias-Wang_1996} J. C. Lagarias and Y. Wang, Self-affine tiles in $\mathbb{R}^n$, {\em Adv. Math.},  {\bf 121} (1996), 21--49.


\bibitem{Li_2008} J.L. Li, Non-spectral problem for a class of planar self-affine measures, {\em J. Funct. Anal.}, {\bf 255} (2008), 3125--3148.


\bibitem{Li_2009} J.L. Li, Non-spectrality of planar self-affine measures with three-element digit set,{\em J. Funct. Anal.}, {\bf 257} (2009),537--552.

\bibitem {Li_2010}  J.L. Li,  On the $\mu_{M,D}$-orthogonal exponentials,  {\em Nonlinear Anal.}, {\bf 73} (2010),940--951.

\bibitem{Li_2011}  J.L. Li, Spectra of a class of self-affine measures, {\em J. Funct. Anal.}, {\bf 260} (2011), 1086--1095.

\bibitem{Li_2015} J.L. Li, Spectral self-affine measures on the spatial Sierpinski gasket, {\em Monatsh. Math.}, {\bf 176} no.2 (2015), 293--322.

\bibitem{Li-Wen_2012}
J.L. Li and Z.Y. Wen , Spectrality of planar self-affine measures with two-element digit set, {\em Sci. China Math.}, {\bf 55} (2012)593--605.

\bibitem{Strichartz_1998}
R.S. Strichartz, Remarks on:``dense analytic subspaces in fractal $L^2$-spaces'', {\em J. Anal. Math.}, {\bf 75} (1998)229--231.

\bibitem{Strichartz_2000}
R.S. Strichartz,  Mock Fourier series and transforms associated with certain Cantor measures, {\em J. Anal. Math.}, {\bf 81} (2000)209--238.




\bibitem{Tao2004} T. Tao, Fuglede's conjecture is false in $5$ or higher dimensions, {\em Math. Res. Lett.}, {\bf 11} (2004) 251--258.


\end{thebibliography}
\end{document}